\documentclass[11pt,
]{amsart}
\usepackage{
amssymb, graphicx, color, cite
}
\usepackage[T1]{fontenc}

\date{\today}
\usepackage[linkcolor=blue, citecolor=black, breaklinks=true]{hyperref}

\usepackage{mathtools}

\mathtoolsset{showonlyrefs}
\renewcommand{\r}[1]{\eqref{#1}}
\numberwithin{equation}{section}%

\newtheorem{theorem}{Theorem}
\newtheorem{lemma}{Lemma}[section]
\newtheorem{proposition}[lemma]{Proposition}
\newtheorem{corollary}[lemma]{Corollary}

\theoremstyle{definition}

\DeclareMathOperator{\tr}{tr}

\DeclareMathOperator{\supp}{supp}

\newcommand{\eps}{\varepsilon}

\newcommand{\R}{{\mathbb R}}
\renewcommand{\S}{{\mathbb S}}
\newcommand{\C}{{\mathbb C}}
\newcommand{\Id}{\mbox{\rm Id}}

\newcommand{\be}[1]{\begin{equation}\label{#1}}
\newcommand{\ee}{\end{equation}}

\renewcommand{\d}{\mathrm{d}}

\renewcommand{\i}{\mathrm{i}}

\newcommand{\bo}{{\partial M}}

\newcommand{\B}{\mathcal{B}}

\newcommand{\e}{\mathsf{e}}

\title[Local rigidity  for the anisotropic Helmholtz equation at fixed energy]{Local rigidity of conformally Euclidean metrics for the anisotropic Helmholtz equation at fixed energy}

\author[S. Mu\~noz-Thon]{Sebasti\'an Mu\~noz-Thon}
\address{Universit\'e Paris-Saclay, Laboratoire de math\'ematiques d'Orsay, 91405, Orsay, France}
\email{sebastian.munoz-thon@universite-paris-saclay.fr}

\author[P. Stefanov]{Plamen Stefanov}
\address{Department of Mathematics, Purdue University, West Lafayette, IN 47907}
\email{stefanov@math.purdue.edu}

\DeclareMathOperator{\diver}{div}

\begin{document}
\begin{abstract}
We prove local rigidity of smooth conformally Euclidean metrics $g_0$ on smooth compact domains $M\subset\R^n$, $n\ge3$, for the anisotropic Helmholtz equation $(-\Delta_g -k^2)u =0$ at a fixed energy $k^2>0$ which is not a Dirichlet eigenvalue of $-\Delta_{g_0}$ in $M$. Any smooth metric sufficiently close to $g_0$ in $H^s(M)$, with integer $s>n/2+1$, and having the same Dirichlet-to-Neumann map, is its pullback by a smooth diffeomorphism fixing the boundary pointwise. The proof uses a Helmholtz embedding to fix the gauge and a compactness argument together with a characterization of the kernel of the linearization.
\end{abstract}

\maketitle

\section{Introduction}\label{sec1} 
Let $M\subset\mathbb R^n$, $n\ge3$, be a smooth compact domain, and let $\mathsf e$ denote the Euclidean metric. Let $k>0$ be fixed.  
For a smooth Riemannian metric $g$ on $M$, we define the Dirichlet-to-Neumann (DN) map, for $f\in C^\infty(\bo)$, by
\[
\Lambda_g(k) f=\left.\partial_{\nu_g}u\right|_{\bo},
\]
where  $u$ is the unique solution of the Dirichlet problem
\begin{align}
(-\Delta_g -k^2)u& =0\quad \text{in }M,\\
u&=f  \quad \text{on }\bo,
\end{align}
assuming that $k^2$ is not a Dirichlet eigenvalue of $-\Delta_g$ in $M$.  
Here,  $\nu_g$ is the outward unit normal with respect to $g$. 
We prove  that for every smooth conformal factor $c$ in $M$, if $g$ is sufficiently close to $g_0\colon=e^{2c}\mathsf e$ in $H^s(M)$, with integer $s>n/2+1$, and if $\Lambda_g(k)=\Lambda_{g_0}(k)$, then $g$ is isometric to $g_0$ by a diffeomorphism fixing the boundary pointwise.

\begin{theorem}\label{thm2}
Let $M\subset\R^n$, $n\ge3$, be a smooth compact domain, let $g_0=e^{2c}\e$ with $c\in C^\infty(M)$, and fix an integer $s>n/2+1$. Fix $k\ge 0$ such that $k^2$ is not a Dirichlet eigenvalue for $-\Delta_{g_0}$ in $M$. 
There exists $\varepsilon_0>0$, depending on $M$, $g_0$, $k$, $s$, such that every smooth Riemannian metric $g$ on $M$ satisfying
\[
\|g-g_0\|_{H^s(M)}<\varepsilon_0
\]
and
\begin{equation}\label{15}
\Lambda_g(k)=\Lambda_{g_0}(k)
\end{equation}
has the form $g=\Phi^*g_0$ for a smooth diffeomorphism $\Phi\colon M\to M$ fixing $\partial M$ pointwise.
\end{theorem}

Theorem~\ref{thm2} extends the result in \cite{MS26} from $k=0$ to $k\ge0$ with a proof following that in \cite{MS26}. For a real energy $\lambda$ below the lowest Dirichlet eigenvalue of $-\Delta_{g_0}$, including $\lambda<0$, local rigidity for $-\Delta_g-\lambda$ follows from \cite[Corollary~1.1]{MS26}. The conformal factor supplied by that corollary equals one by unique continuation, since its difference from one satisfies a linear elliptic equation with zero Cauchy data. Recently, OpenAI announced \cite{OpenAI26} a solution of the anisotropic  Calder\'on problem for $k=0$ but at the time of writing this paper, the proof is still being checked for accuracy.

Ceki\'c, Krupchyk, Sahoo, and Uhlmann \cite{CKSU26} showed recently that, for non-trapping metrics with strictly convex boundary, equality of Helmholtz DN maps at a sufficiently large frequency implies equality of lens data. The frequency threshold depends on the two metrics; metric uniqueness then follows under assumptions ensuring lens rigidity. In fact, under these lens rigidity assumptions, their result is equivalent to uniqueness, modulo the gauge, given the DN maps over a sequence $k_j\to\infty$.

For a smooth metric $g$ equal to $\e$ outside a ball, let $S_g(k)$ denote the scattering matrix of $-\Delta_g$ at energy $k^2$, see, e.g., \cite[Sections~4.1 and~4.4]{DZ-book}. Equality of these matrices is equivalent to equality of the scattering amplitudes for all incident and observation directions. Theorem~\ref{thm2} has the following consequence.

\begin{corollary}\label{cor2}
Let $B\subset\R^n$, $n\ge3$, be a ball, and let $g_0=e^{2c}\e$, where $c\in C_0^\infty(B)$. Fix $k>0$ and an integer $s>n/2+1$. There exists $\eps_0>0$, depending on $B,g_0,k,s$, such that every smooth Riemannian metric $g$ on $\R^n$ satisfying
\[
g=\e\quad\text{outside }B,\qquad \|g-g_0\|_{H^s(\R^n)}<\eps_0,
\]
and $S_g(k)=S_{g_0}(k)$, has the form $g=\Phi^*g_0$, where $\Phi\colon\R^n\to\R^n$ is a smooth diffeomorphism equal to the identity outside $B$.
\end{corollary}

Novikov \cite{Novikov88} and Ramm \cite{Ramm88} establish uniqueness for smooth compactly supported  potentials in dimensions $n\ge3$ from the  scattering amplitude at any fixed positive energy. The proof is based on the complex geometric optics solutions in \cite{SylvesterU87}. 
By the Liouville transform, this also yields uniqueness for conformally Euclidean metrics equal to $\e$ outside a ball. Uniqueness at fixed energy 
was proved for real exponentially decreasing  potentials in \cite{Novikov94,UhlmannVasy02} and for  electrics potentials and magnetic fields in \cite{EskinRalston95}. In Corollary~\ref{cor2}, only the background $g_0$ is assumed conformally Euclidean; the second metric $g$ is allowed to be an arbitrary sufficiently small anisotropic perturbation.

\subsection*{Acknowledgments}
The authors used ChatGPT (OpenAI) for assistance with checking computations. The authors take responsibility for the final mathematical statements and the proofs. Sebasti\'an Mu\~noz-Thon was supported by the European Research Council (ERC) under the European Union's Horizon Europe research and innovation programme (Grant agreement no. 101162990 -- ADG). Plamen Stefanov was partly supported by National Science Foundation Grant DMS-2452757.

\section{Proof of Theorem~\ref{thm2}}
\label{sec3}

\subsection*{Preliminaries}
We fix $M, g_0, k, s$ as in Theorem~\ref{thm2} and assume that $g$ is a smooth Riemannian metric satisfying \r{15} and sufficiently close to $g_0$ in $H^s(M)$.
For a symmetric matrix field $A$, write $L_Au=\partial_i(A^{ij}\partial_j u)$. For a metric $g$, write $|g|=\det(g_{ij})$. Its associated conductivity is
\begin{equation}\label{16}
A_g=|g|^{1/2}\,g^{-1},\qquad L_{A_g}=|g|^{1/2}\,\Delta_g.
\end{equation}
Fix a smooth compactly supported extension of $c$ to $\R^n$ and continue to write $g_0=e^{2c}\e$ for the extended background. Set
\begin{equation}\label{17}
a=e^{(n-2)c/2},\qquad A_0\coloneq A_{g_0}=a^2\Id.
\end{equation}

Under the DN equality \r{15}, boundary determination \cite[Theorem~8.4]{DosSantosKSU09} implies equality of the induced boundary metrics. Green's formula therefore yields the Helmholtz Alessandrini identity
\begin{equation}\label{19}
\int_M\bigl[(A_g-A_0)^{ij}(\partial_i u)\,(\partial_j v)-k^2\bigl(|g|^{1/2}-|g_0|^{1/2}\bigr)uv\bigr]\,\d x=0
\end{equation}
for $u,v\in C^\infty(M)$ satisfying $(\Delta_g+k^2)u=0$ and $(\Delta_{g_0}+k^2)v=0$ in $M$.

\subsection{The Helmholtz embedding}\label{sec3a}

The next proposition holds for general smooth metrics $g$. 
\begin{proposition} \label{prop:embedding}
There exist an integer $N$, an open neighborhood $U$ of $M$, and a smooth embedding $F\colon U\to\R^N$ satisfying
\[
(-\Delta_{g_0}-k^2)F^\alpha=0.
\]
\end{proposition}

\begin{proof}
Choose concentric balls $B_1,B_2$ with $M\subset B_1\Subset B_2$ such that $k^2$ is not a Dirichlet eigenvalue of $-\Delta_{g_0}$ in $B_2$, and set $K=\overline{B_1}$. Such an outer radius exists by strict domain monotonicity of the Dirichlet eigenvalues. Fix a nonempty open subset $\Gamma\subset\partial B_2$, and let $u_f$ solve $(-\Delta_{g_0}-k^2)u_f=0$ in $B_2$ with boundary value $f\in C_0^\infty(\Gamma)$, real-valued. We adapt \cite[Propositions~A.1 and~A.5]{LLS-Poisson20}, using the Runge approximation theorem \cite[Proposition~A.2]{LLS-Poisson20} for this operator.

First we show that solutions to the Helmholtz problem separate points. Let $x\ne y$ be points in $B_2$. Choose disjoint small balls $D_1,D_2\Subset B_2$ centered at $x,y$, with $B_2\setminus\overline{D_1\cup D_2}$ connected. Solving the Helmholtz equation in $D_1$ with boundary value $1$ produces a solution $v_1$ with $v_1(x)\ne0$ when $D_1$ is sufficiently small, as follows by rescaling. Runge approximation of $v_1$ on $D_1$ and $0$ on $D_2$, followed by Sobolev embedding, yields $u_f(x)\ne u_f(y)$.

For the differentials, fix $p\in B_2$ and $\xi\in T_p^*B_2$. On a sufficiently small ball $B(p,\eps)$, let $v_\eps$ solve the Helmholtz equation with boundary value $\xi\cdot(x-p)$. Rescaling by $\eps^{-1}v_\eps(p+\eps y)$ makes the zeroth-order term $-\eps^2k^2$, and elliptic estimates imply
\[
|\d v_\eps(p)-\xi|\le C(\eps+\eps^2k^2)|\xi|.
\]
For each fixed $\eps$, Runge approximation in $H^{n/2+2}(B(p,\eps))$, hence in $C^1$, approximates $\d v_\eps(p)$ by differentials $\d u_f(p)$. Letting $\eps\to0$ shows that the range of $f\mapsto\d u_f(p)$ is dense in $T_p^*B_2$. Being a linear subspace of a finite-dimensional space, it is closed, so every $\xi$ is attained.

Given $x\in K$, choose solutions $u_x^1,\ldots,u_x^n$ whose differentials form a basis at $x$. By the inverse function theorem, $U_x=(u_x^1,\ldots,u_x^n)$ is a diffeomorphism on a neighborhood $V_x$ of $x$. Choose a finite cover $K\subset V_{x_1}\cup\cdots\cup V_{x_m}$ and set
\[
F_0=(u_{x_1}^1,\ldots,u_{x_1}^n,\ldots,u_{x_m}^1,\ldots,u_{x_m}^n).
\]
Then $F_0$ is an immersion on $K$ and is injective on each $V_{x_j}$.

By the Lebesgue number lemma \cite[Chapter~XI, Theorem~4.5]{Dugundji66}, $F_0$ separates all pairs in $K$ with $0<|x-y|<\lambda$ for some $\lambda>0$. The remaining pairs form the compact set
\[
\{(x,y)\in K\times K:|x-y|\ge\lambda\}.
\]
For each such pair, choose a Helmholtz solution separating its points. By continuity and compactness, finitely many of these solutions $h^1,\ldots,h^q$ suffice. Thus $F=(F_0,h^1,\ldots,h^q)$ is an injective immersion on $K$. Since $K$ is compact, $F$ is a homeomorphism onto its image \cite[Chapter~XI, Theorem~2.1]{Dugundji66}, hence a smooth embedding. Restricting to $U=B_1$ proves the proposition.
\end{proof}

\subsection{The Helmholtz gauge}\label{sec3b}
The next lemma extends the gauge construction in \cite[Lemma~3.2]{MS26} to the Helmholtz equation.

\begin{lemma}\label{lemma6}
Let $M,g_0,k,s$ be as in Theorem~\ref{thm2}, and let $F$ be the embedding in Proposition~\ref{prop:embedding}. Let $g$ be a smooth metric sufficiently close to $g_0$ in $H^s(M)$. Then there exist a smooth diffeomorphism $\Phi\colon M\to M$ fixing $\partial M$ pointwise and $Z\in C^\infty(M;\R^N)$ with the following properties. Set
\begin{equation}\label{20}
\overline{g}=\Phi_*g,\qquad B=a^{-2}(A_{\overline{g}}-A_0).
\end{equation}
Then
\begin{equation}\label{21}
(\Delta_{\overline{g}}+k^2)(F+Z)=0,\qquad Z|_{\partial M}=0,\qquad
\partial_jF\cdot Z=0,\quad j=1,\ldots,n,
\end{equation}
and
\begin{equation}\label{22}
\begin{split}
\|\Phi-\Id\|_{H^{s+1}(M)}+\|B\|_{H^s(M)}+\|Z\|_{H^{s+1}(M)}&\le C\|g-g_0\|_{H^s(M)},\\
\|Z\|_{H^1(M)}&\le C\|B\|_{L^2(M)}.
\end{split}
\end{equation}
The constant $C$ and the required smallness of $\|g-g_0\|_{H^s(M)}$ depend only on $M,g_0,k,s,F$.
\end{lemma}

\begin{proof}
Since $k^2$ is not a Dirichlet eigenvalue of $-\Delta_{g_0}$, it is not a Dirichlet eigenvalue of $-\Delta_g$ for $g$ sufficiently close to $g_0$. We can therefore solve
\[
(\Delta_g+k^2)F_g^\alpha=0\quad\text{in }M,\qquad F_g^\alpha|_{\partial M}=F^\alpha|_{\partial M},\qquad \alpha=1,\ldots,N.
\]
By elliptic regularity up to the boundary, $F_g\in C^\infty(M;\R^N)$. Since $F$ solves the background equation, we have
\[
(\Delta_{g_0}+k^2)(F_g-F)=(\Delta_{g_0}-\Delta_g)F_g,\qquad (F_g-F)|_{\partial M}=0.
\]
By standard elliptic estimates and Sobolev product estimates,
\[
\begin{split}
\|F_g-F\|_{H^{s+1}}&\le C\|(\Delta_{g_0}-\Delta_g)F_g\|_{H^{s-1}}\\
&\le C\|g-g_0\|_{H^s}\bigl(1+\|F_g-F\|_{H^{s+1}}\bigr),
\end{split}
\]
where we used that $F$ is fixed and smooth. 
All norms here and below are over $M$. For $g$ sufficiently close to $g_0$, we can absorb the last term to get
\begin{equation}\label{23}
\|F_g-F\|_{H^{s+1}}\le C\|g-g_0\|_{H^s}.
\end{equation}

Let $\pi$ be the Euclidean nearest-point projection onto $F(U)$, defined and smooth near $F(M)$. By the preceding estimate and Sobolev embedding, $F_g(M)$ lies in its domain for $g$ sufficiently close to $g_0$. Set
\begin{equation}\label{24}
\Phi=F^{-1}\circ\pi\circ F_g.
\end{equation}
As in the proof of \cite[Lemma~3.2]{MS26}, $\Phi$ is a smooth diffeomorphism of $M$ onto itself fixing $\partial M$ pointwise. Set
\begin{equation}\label{26}
Z=F_g\circ\Phi^{-1}-F.
\end{equation}
The Sobolev estimates in that proof also establish the first estimate of the lemma. By the definition of $\pi$, $Z$ is normal to $F(U)$, and $Z=0$ on $\partial M$ because $F_g=F$ there. Invariance of the Helmholtz equation under diffeomorphisms implies $(\Delta_{\overline{g}}+k^2)(F+Z)=0$.

It remains to estimate $Z$ in $H^1$. Since $(\Delta_{g_0}+k^2)F=0$, we have
\begin{equation}\label{27}
\bigl(L_{A_{\overline{g}}}+k^2|\overline{g}|^{1/2}\bigr)Z =-\diver(a^2B\nabla F)-k^2\bigl(|\overline{g}|^{1/2}-|g_0|^{1/2}\bigr)F,
\qquad Z|_{\partial M}=0.
\end{equation}
Since $A_{\overline{g}}=a^2(\Id+B)$ and $n\ge3$,
\[
|\overline{g}|^{1/2}=|g_0|^{1/2}\,\det(\Id+B)^{1/(n-2)}.
\]
For $B$ small in $L^\infty$, this implies
\[
\bigl\||\overline{g}|^{1/2}-|g_0|^{1/2}\bigr\|_{L^2}\le C\|B\|_{L^2}.
\]
By the assumption on $k^2$, the operator $L_{A_0}+k^2|g_0|^{1/2}$ is invertible from $H_0^1(M)$ to $H^{-1}(M)$. The $H^s$ bound for $B$ and Sobolev embedding show that the corresponding operator for $\overline{g}$ is close to it in the operator norm from $H_0^1(M)$ to $H^{-1}(M)$. Its inverse is therefore uniformly bounded. Applying it to the equation for $Z$, we get
\[
\|Z\|_{H^1} \le C\left(\|a^2B\nabla F\|_{L^2}+k^2\bigl\|(|\overline{g}|^{1/2}-|g_0|^{1/2})F\bigr\|_{L^2}\right)\le C\|B\|_{L^2}.
\]
\end{proof}

\subsection{Boundary jets}\label{sec3c}

\begin{lemma}\label{lemma7}
Let $g$ and $\overline{g}$ be as in Lemma~\ref{lemma6}, and assume that $\Lambda_g(k)=\Lambda_{g_0}(k)$. Then $\overline{g}-g_0$ is flat at $\partial M$.
\end{lemma}

\begin{proof}
We work near an arbitrary boundary point, and write $\equiv$ for equality of boundary jets. We will use the boundary jet recovery  Theorem~8.4 in \cite{DosSantosKSU09} for general elliptic second order operators which allows a conformal gauge because it treats both the metric and the potential as unknowns; while in our case, the potential is fixed to be $-k^2$. Applying it to the pairs $(g,-k^2)$ and $(g_0,-k^2)$, and composing the conformal gauges by multiplying their factors, we obtain a local diffeomorphism $\psi$ fixing the boundary pointwise and a smooth $f>0$, with $f=1$ and $\d f=0$ on the boundary, such that, with $\widehat g=\psi^*g$,
\[
\widehat g\equiv f^{4/(n-2)}g_0,\qquad
-k^2\equiv f^{-4/(n-2)}\left(-k^2-\frac{\Delta_{g_0}f}{f}\right).
\]
Then
\[
\Delta_{g_0}f\equiv k^2\left(f^{(n+2)/(n-2)}-f\right).
\]
This is an elliptic semilinear PDE (which we are not trying to solve). 
Successive normal differentiation, using $f=1$ and $\d f=0$ on the boundary, shows that $f-1$ is flat, hence so is $\widehat g-g_0$.

Set $W=F_g\circ\psi-F$. Then
\begin{equation}\label{28}
\begin{gathered}
(\Delta_{\widehat g}+k^2)W=(\Delta_{g_0}-\Delta_{\widehat g})F,\\
W|_{\partial M}=\partial_{\nu_{\widehat g}}W|_{\partial M}=0.
\end{gathered}
\end{equation}
The Neumann condition follows from the DN equality and $\widehat g=g_0$ on the boundary. Since the right-hand side is flat, normal differentiation shows that $W$ is flat as well. The projection argument in \cite[Lemma~3.3]{MS26} now shows that $\Phi\circ\psi-\Id$ is flat, too. Since $(\Phi\circ\psi)_*\widehat g=\overline{g}$, the lemma follows.
\end{proof}

Lemma~\ref{lemma7} allows us to extend $\overline{g}$ by $g_0$ outside $M$, and $B$ by zero; then $B\in C_0^\infty(\R^n)$. By \r{22},
\begin{equation}\label{29}
\|B\|_{H^s(\R^n)}\le C\|g-g_0\|_{H^s(M)}.
\end{equation}

The estimate in \cite[Lemma~3.4]{MS26} also holds for the Helmholtz gauge.

\begin{lemma}\label{lemma8}
For $B$ as above,
\begin{equation}\label{30}
\|\diver B\|_{L^2(\R^n)}\le C\|B\|_{L^2(\R^n)},
\end{equation}
where $C$ depends only on $M,g_0,k,s,F$.
\end{lemma}

\begin{proof}
By Lemma~\ref{lemma6},
\begin{equation}\label{31}
L_{A_{\overline{g}}}Z
=-\partial_i(a^2B^{ij}\partial_jF)-k^2|\overline{g}|^{1/2}\,Z 
-k^2\bigl(|\overline{g}|^{1/2}-|g_0|^{1/2}\bigr)F.
\end{equation}
Take scalar products with $\partial_\ell F$. The term $-k^2|\overline{g}|^{1/2}Z$ makes no contribution, since $Z\cdot\partial_\ell F=0$. The tangential calculation in \cite[Lemma~3.4]{MS26} then gives a system for $\partial_i(a^2B^{ij})$ with the same coefficient matrix $\partial_\ell F\cdot\partial_jF-Z\cdot\partial_j\partial_\ell F$. 
The only additional term on the right is
\[
-k^2\bigl(|\overline{g}|^{1/2}-|g_0|^{1/2}\bigr)\,\partial_\ell F\cdot F.
\]
Since $F$ is fixed and smooth, the estimate for the volume-density difference in the proof of Lemma~\ref{lemma6} bounds its $L^2(M)$ norm by $C\|B\|_{L^2(M)}$. Since $F$ is an embedding and $Z$ is small in $L^\infty$, the matrix is uniformly invertible, as in \cite[Lemma~3.4]{MS26}. Therefore,
\[
\begin{split}
\|\diver(a^2B)\|_{L^2}&\le C\left(\|B\|_{L^2}+\|Z\|_{H^1}+k^2\bigl\||\overline{g}|^{1/2}-|g_0|^{1/2}\bigr\|_{L^2}\right) \le C\|B\|_{L^2},
\end{split}
\]
where we used Lemma~\ref{lemma6} again, and all norms are in $M$. Finally, $\diver B=a^{-2}\diver(a^2B)-2B\nabla\log a$. Since $B$ extends smoothly by zero by Lemma~\ref{lemma7}, the estimate holds on $\R^n$ as well.
\end{proof}

\subsection{The Liouville transform}\label{sec3d}
Extend $\overline{g}$ by the fixed extended background $g_0$ and $B$ by zero outside $M$, as in Lemma~\ref{lemma7}. Then $A_{\overline{g}}=a^2(\Id+B)$ on $\R^n$. Set
\begin{equation}\label{34}
b_2=\|B\|_{L^2},\qquad b_s=\|B\|_{H^s},\qquad b_\infty=\|B\|_{L^\infty}.
\end{equation}
All norms below are over $\R^n$ unless another domain is indicated. Sobolev embedding and interpolation imply
\begin{equation}\label{35}
b_\infty+\|\diver B\|_{L^\infty}\le Cb_s
\end{equation}
and
\begin{equation}\label{36}
b_\infty\le C\|B\|_{H^{s-1}}\le Cb_2^{1/s}b_s^{1-1/s}.
\end{equation}
As in \cite[Proposition~2.5]{MS26}, the continuation argument will require $\tau b_\infty$ to be small; the interpolation estimate will allow us to choose the frequency cutoff within this range.

The substitution $u=v/a$ in $(\Delta_{\overline{g}}+k^2)u=0$ leads to the following version of \cite[Lemma~3.5]{MS26}. Set
\begin{equation}\label{18}
q=\frac{\Delta a}{a}-k^2e^{2c},\qquad q+k^2\in C_0^\infty(\R^n),
\end{equation}
where $\Delta$ is the Euclidean Laplacian.

\begin{lemma}\label{lemma9}
Set
\begin{equation}\label{37}
q_B=a^{-1}\diver(B\nabla a)-k^2e^{2c}\left(\det(\Id+B)^{1/(n-2)}-1\right).
\end{equation}
Then, for $v\in C^\infty(\R^n)$,
\begin{equation}\label{38}
a^{-1}\bigl(L_{A_{\overline{g}}}+k^2|\overline{g}|^{1/2}\bigr)(v/a)
=(\Delta-q)v+\diver(B\nabla v)-q_Bv.
\end{equation}
Furthermore,
\begin{equation}\label{39}
\|\diver B\|_{L^2}+\|q_B\|_{L^2}\le Cb_2,\qquad
\|\diver B\|_{L^\infty}+\|q_B\|_{L^\infty}\le Cb_s.
\end{equation}
The constant $C$ depends only on $M,g_0,k,s,F$ and the fixed extension of $c$.
\end{lemma}

\begin{proof}
By the calculation in \cite[Lemma~3.5]{MS26},
\begin{equation}\label{40}
a^{-1}L_{A_0}(v/a)=\Delta v-\frac{\Delta a}{a}v,
\end{equation}
and the conductivity perturbation contributes $\diver(B\nabla v)-a^{-1}\diver(B\nabla a)v$. The additional Helmholtz term is determined by
\[
a^{-2}|\overline{g}|^{1/2}=e^{2c}\det(\Id+B)^{1/(n-2)},
\]
which proves the identity in the lemma.

For $B$ small in $L^\infty$,
\[
\bigl\|\det(\Id+B)^{1/(n-2)}-1\bigr\|_{L^p}\le C\|B\|_{L^p},\qquad p=2,\infty.
\]
The differential term in $q_B$ is $\nabla\log a\cdot\diver B+a^{-1}B^{ij}\partial_i\partial_j a$. The estimates therefore follow from Lemma~\ref{lemma8} and Sobolev embedding.
\end{proof}

The potential $q$ equals $-k^2$ outside a compact set. To construct the background complex geometric optics solutions, we will cut it off outside a neighborhood of the region under consideration.

\subsection{Background complex geometric optics solutions}\label{sec3e}

\begin{lemma}\label{lemma10}
Fix a closed ball $\B$ containing $M$ in its interior and an integer $\ell\ge0$. There are constants $\tau_*$, $C>0$ such that, if $\zeta\in\C^n$, $\zeta\cdot\zeta=0$, and $|\zeta|\ge \tau_*$, then there is a smooth function on $\R^n$ of the form
\begin{equation}\label{42}
u_\zeta^0=a^{-1}e^{x\cdot\zeta}(1+r_\zeta^0)
\end{equation}
satisfying
\[
(\Delta_{g_0}+k^2)u_\zeta^0=0\quad\text{in a neighborhood of }\B
\]
and
\begin{equation}\label{43}
\|r_\zeta^0\|_{C^\ell(\B)}\le\frac{C}{|\zeta|}.
\end{equation}
The constants depend only on the fixed extended background $g_0$, $k$, $\B$, $\ell$, and the fixed cutoff in the proof.
\end{lemma}

\begin{proof}
By the Liouville identity \r{38} with $B=0$ and $\zeta\cdot\zeta=0$, the ansatz \r{42} satisfies
\[
a^{-1}|g_0|^{1/2}e^{-x\cdot\zeta}(\Delta_{g_0}+k^2)u_\zeta^0=P_\zeta r_\zeta^0-q(1+r_\zeta^0),\qquad P_\zeta=\Delta+2\zeta\cdot\nabla.
\]
Choose $\chi\in C_0^\infty(\R^n)$ equal to $1$ near $\B$, and set $\tilde q=\chi q$. It suffices to solve
\[
P_\zeta r_\zeta^0=\tilde q(1+r_\zeta^0)\quad\text{on }\R^n.
\]
The Sylvester--Uhlmann construction \cite{SylvesterU87}, with the $C^\ell$ remainder estimate in \cite[Lemma~3.6]{MS26}, yields a smooth solution satisfying \r{43} for $|\zeta|$ sufficiently large. Since $\tilde q=q$ near $\B$, the first identity proves the assertion.
\end{proof}

For $\tau\ge1$, set
\begin{equation}\label{47}
X_\tau(v)=\|v\|_{L^2}+\tau^{-1}\|\nabla v\|_{L^2}+\tau^{-2}\|\nabla^2v\|_{L^2}.
\end{equation}

\begin{corollary}\label{cor1}
There are constants $\tau_0\ge1$ and $C>0$, depending only on $\B,n$ and $\|q\|_{L^\infty(\B)}$, such that, for $v\in H^2(\R^n)$ supported in $\B$ and $\zeta\in\C^n$ with $\zeta\cdot\zeta=0$ and $|\zeta|=\sqrt2\tau$, we have
\begin{equation}\label{48}
X_\tau(v)\le\frac{C}{\tau}\|(P_\zeta-q)v\|_{L^2},\qquad \tau\ge \tau_0.
\end{equation}
\end{corollary}

\begin{proof}
The proof of \cite[Corollary~3.7]{MS26} requires only a bound for $q$ on the support of $v$. Indeed, for $\tau\ge1$,
\[
X_\tau(v)\le\frac{C}{\tau}\|P_\zeta v\|_{L^2}\le\frac{C}{\tau}\|(P_\zeta-q)v\|_{L^2}+\frac{C\|q\|_{L^\infty(\B)}}{\tau}X_\tau(v).
\]
Choose $\tau_0$ so that $C\|q\|_{L^\infty(\B)}/\tau_0\le1/2$ and absorb the last term to obtain \r{48}.
\end{proof}

The continuation argument of \cite[Lemma~3.8]{MS26} applies to the Helmholtz equation as follows.

\begin{lemma}\label{lemma11}
Let $g,\overline{g}$ be as in Lemma~\ref{lemma7}. After increasing $\tau_0$ from Corollary~\ref{cor1}, there are constants $c_0,C>0$ such that the following holds. Let $\zeta\cdot\zeta=0$, $|\zeta|=\sqrt2\tau$, and assume
\begin{equation}\label{49}
b_s\le c_0,\qquad \tau\ge \tau_0,\qquad \tau b_\infty\le c_0.
\end{equation}
Let $u_\zeta^0$ be the background solution of Lemma~\ref{lemma10}, with the $C^2(\B)$ bound in \r{43}. The solution $u_\zeta$ of $(\Delta_{\overline{g}}+k^2)u_\zeta=0$ in $M$ with boundary value $u_\zeta^0$ extends by $u_\zeta^0$ outside $M$ to a smooth function on $\R^n$ solving the same equation on a neighborhood of $\B$. It has the form
\begin{equation}\label{50}
u_\zeta=a^{-1}e^{x\cdot\zeta}(1+r_\zeta^0+r_\zeta),\qquad r_\zeta\in C_0^\infty(\R^n),\qquad \supp r_\zeta\subset M,
\end{equation}
with
\begin{equation}\label{51}
X_\tau(r_\zeta)\le C\tau b_2.
\end{equation}
The constants depend only on $M,g_0,k,s,F,\B$, the fixed extension of $c$, and the cutoff in Lemma~\ref{lemma10}.
\end{lemma}

\begin{proof}
The argument of \cite[Lemma~3.8]{MS26} applies with the potentials $q$ and $q_B$ defined above, using Lemmas~\ref{lemma9} and~\ref{lemma10} and Corollary~\ref{cor1}.
\end{proof}

\subsection{Estimates on the Fourier transform}\label{sec3g}
The following polarization lemma from the previous paper applies without change.

\begin{lemma}[{\cite[Lemma~3.9]{MS26}}]\label{lemma12}
For every $\omega\in\S^{n-1}$, one can choose $L=L(n)$ pairs $(z_\ell,w_\ell)\in\C^n\times\C^n$ such that
\[
z_\ell\cdot z_\ell=w_\ell\cdot w_\ell=0,\qquad z_\ell+w_\ell=-\i\omega,\qquad |z_\ell|=|w_\ell|=\sqrt2,
\]
and, for every complex symmetric matrix $S$,
\begin{equation}\label{55}
|S|\le C\left(\sum_{\ell=1}^L|S(z_\ell,w_\ell)|+|S\omega|\right).
\end{equation}
Here $S(z,w)=S^{ij}z_i w_j$, $|S|$ is the Frobenius norm, and $C$ depends only on $n$. The vectors may be chosen piecewise smoothly in $\omega$.
\end{lemma}

For $S=\widehat B(\xi)$ and $\omega=\xi/|\xi|$, the additional term in \r{55} is controlled by Lemma~\ref{lemma8}. Indeed, $\widehat{\diver B}(\xi)=\i\widehat B(\xi)\xi$, so Plancherel's identity and \r{30} imply
\begin{equation}\label{56}
\left\|\mathbf1_{\{|\xi|>T\}}\widehat B(\xi)\frac{\xi}{|\xi|}\right\|_{L^2}\le\frac{C}{T}\|\diver B\|_{L^2}\le\frac{C}{T}b_2,\qquad T>0.
\end{equation}

We next estimate the errors in the background complex geometric optics amplitudes.

\begin{lemma}[\!{\cite[Lemma~3.10]{MS26}}]\label{lemma13}
Let $e(x,\xi)=(e^{ij}(x,\xi))$ be measurable in $\xi$ and $C^{n+1}$ in $x$, with support in a fixed compact set in $x$, and suppose that
\begin{equation}\label{57}
|\partial_x^\alpha e(x,\xi)|\le C|\xi|^{-1},\qquad |\alpha|\le n+1,\quad |\xi|\ge \tau_0.
\end{equation}
For a square-integrable matrix field $B$, define
\begin{equation}\label{58}
E_TB(\xi)=\mathbf1_{\{|\xi|>T\}}\int_{\R^n}e^{-\i x\cdot\xi}e^{ij}(x,\xi)B^{ij}(x)\,\d x,\qquad T\ge \tau_0.
\end{equation}
Then
\begin{equation}\label{59}
\|E_TB\|_{L^2}\le\frac{C}{T}\|B\|_{L^2},
\end{equation}
where $C$ depends only on $n$, the fixed compact support, and the constant in \r{57}.
\end{lemma}

We next adapt the Fourier estimate of \cite[Lemma~3.11]{MS26} to the Helmholtz equation.

\begin{lemma}\label{lemma14}
Let $g,\overline{g}$ be as in Lemma~\ref{lemma7}, with $B$ as in \r{20}. After increasing $\tau_0$ and decreasing $c_0$ from Lemma~\ref{lemma11}, if necessary, we have
\begin{equation}\label{60}
\big\|\mathbf1_{\{|\xi|>T\}}\widehat B\big\|_{L^2}\le C\bigl(R^{n/2+1}b_2^2+T^{-1}b_2+R^{-s}b_s\bigr)
\end{equation}
whenever
\begin{equation}\label{61}
b_s\le c_0,\qquad R\ge T\ge \tau_0,\qquad Rb_\infty\le c_0.
\end{equation}
The constant $C$ has the same dependence as in Lemma~\ref{lemma11} and is independent of $B,T,R$.
\end{lemma}

\begin{proof}
Write $\xi=\rho\omega$ and take $\zeta=\rho z_\ell$, $\eta=\rho w_\ell$ with the vectors of Lemma~\ref{lemma12}. Choose the background solutions with the $C^{n+2}(\B)$ bound in Lemma~\ref{lemma10}, increasing $\tau_0$ if necessary. For $T<\rho\le R$, the assumptions of Lemma~\ref{lemma11} hold. Applying \r{19} to $\overline{g}$, by invariance of the DN map, we obtain
\begin{equation}\label{62}
\int_M a^2B^{ij}(\partial_i u_\zeta^0)\,(\partial_j u_\eta^0)\,\d x
=-\int_M a^2B^{ij}(\partial_i(u_\zeta-u_\zeta^0))\,(\partial_j u_\eta^0)\,\d x +k^2\int_M\bigl(|\overline{g}|^{1/2}-|g_0|^{1/2}\bigr)u_\zeta u_\eta^0\,\d x.
\end{equation}
The gradient terms are treated as in \cite[Lemma~3.11]{MS26}. We estimate the additional term, divided by $\rho^2$, after writing $u_\zeta=u_\zeta^0+(u_\zeta-u_\zeta^0)$. By the estimate for the volume-density difference in the proof of Lemma~\ref{lemma6},
\[
\bigl\||\overline{g}|^{1/2}-|g_0|^{1/2}\bigr\|_{L^2}\le Cb_2.
\]
This density difference is supported in $M$ by Lemma~\ref{lemma7}.

For the two background solutions, use $\zeta+\eta=-\i\xi$ and choose $\chi\in C_0^\infty(\B^{\rm int})$ equal to $1$ near $M$. The amplitude multiplying the density difference is
\[
\frac{k^2\chi(x)}{\rho^2a^2}(1+r_\zeta^0)(1+r_\eta^0)
=O_{C^{n+1}}(k^2\rho^{-2}),\qquad \rho\ge\tau_0.
\]
Since $\rho^{-2}\le T^{-1}\rho^{-1}$ for $\rho>T$, the scalar version of Lemma~\ref{lemma13}, applied to $T$ times this amplitude with the cutoff $\rho>T$, implies
\[
\left\|\mathbf1_{\{|\xi|>T\}}\frac{k^2}{\rho^2}\int_M\bigl(|\overline{g}|^{1/2}-|g_0|^{1/2}\bigr)u_\zeta^0u_\eta^0\,\d x\right\|_{L^2_\xi}
\le Ck^2T^{-2}b_2.
\]
For the part containing the correction, cancellation of the exponentials and Lemma~\ref{lemma11} imply
\[
\begin{split}
\frac{k^2}{\rho^2}\left|\int_M\bigl(|\overline{g}|^{1/2}-|g_0|^{1/2}\bigr)(u_\zeta-u_\zeta^0)u_\eta^0\,\d x\right|
\le\frac{Ck^2}{\rho^2}\bigl\||\overline{g}|^{1/2}-|g_0|^{1/2}\bigr\|_{L^2}\|r_\zeta\|_{L^2} \le Ck^2\rho^{-1}b_2^2.
\end{split}
\]
Its $L^2$ norm on $T<|\xi|\le R$ is therefore at most $Ck^2R^{n/2+1}b_2^2$, since $\rho\ge1$. As $T\ge1$ and $k$ is fixed, both contributions have the bounds required in \r{60}, with constants depending also on $k$. The polarization and Sobolev tail arguments of \cite[Lemma~3.11]{MS26} complete the proof.
\end{proof}

\subsection{Compactness}\label{sec3h} 
If the local rigidity statement of Theorem~\ref{thm2} fails, we have a sequence of nonzero perturbations $B_j$ with $\|B_j\|_{H^s}\to0$. We now prove that their $L^2$-normalizations have a  convergent subsequence.

\begin{lemma}\label{lemma15}
Let $B_j\ne0$ be conductivity perturbations satisfying the hypotheses of Lemma~\ref{lemma14}, with constants $C,c_0,\tau_0$ independent of $j$, and write $b_{2,j},b_{s,j},b_{\infty,j}$ for their norms as in \r{34}. If $b_{s,j}\to0$, then a subsequence of $H_j=B_j/b_{2,j}$ converges  in $L^2(\R^n)$ to a tensor $H$ satisfying
\[
\supp H\subset M,\qquad \|H\|_{L^2}=1.
\]
\end{lemma}

\begin{proof}
The compactness argument of \cite[Lemma~3.12]{MS26} applies, using Lemma~\ref{lemma14} and the interpolation estimate \r{36}.
\end{proof}

\subsection{The linearized problem}\label{sec3i}
We identify below covariant and contravariant tensors using $\e$ and denote the Euclidean tensor by $\e$ in either case. 
For a smooth positive function $a$, and a vector field $X$, define the normalized infinitesimal gauge by
\begin{equation}\label{72}
\mathcal G_aX=DX+DX^\top-\bigl(\diver X+2X\cdot\nabla\log a\bigr)\e.
\end{equation}
As shown in the proof of Lemma~\ref{lemma16}, $a^2\mathcal G_aX$ is the variation of $A_0=a^2\Id$ under pushforward by the flow of $X$.

\begin{lemma}\label{lemma16}
Let $X\in H_0^1(M;\R^n)$, set $H=\mathcal G_aX$, and let $u,v\in C^\infty(M)$ satisfy $(\Delta_{g_0}+k^2)u=(\Delta_{g_0}+k^2)v=0$. Then
\begin{equation}\label{73}
\int_M\left(a^2H^{ij}(\partial_i u)\,(\partial_j v)-\frac{k^2}{n-2}|g_0|^{1/2}(\tr H)uv\right)\,\d x=0,
\end{equation}
and
\begin{equation}\label{74}
\begin{split}
\bigl(L_{A_0}+k^2|g_0|^{1/2}\bigr)(-X\cdot\nabla u)
&=-\diver(a^2H\nabla u) -\frac{k^2}{n-2}|g_0|^{1/2}(\tr H)u
\end{split}
\end{equation}
in the distributional sense. The same identities hold on $\R^n$ for compactly supported $X\in H^1(\R^n;\R^n)$ and $u,v\in C^\infty(\R^n)$ satisfying the background Helmholtz equation on a neighborhood of $\supp X$.
\end{lemma}

\begin{proof}
First let $X$ be smooth and compactly supported, let $\varphi_t$ be its flow, and set $g_t=(\varphi_t)_*g_0$. By \cite[Lemma~3.13]{MS26}, the conductivity variation is $\left.\frac{\d}{\d t}A_{g_t}\right|_{t=0}=a^2H$. The volume density transforms by
\[
|g_t|^{1/2}=\left(\frac{|g_0|^{1/2}}{\det D\varphi_t}\right)\circ\varphi_t^{-1},
\]
so
\[
\left.\frac{\d}{\d t}|g_t|^{1/2}\right|_{t=0}
=-\diver\bigl(|g_0|^{1/2}X\bigr)
=\frac{|g_0|^{1/2}}{n-2}\tr H.
\]
The last equality follows from \r{72} and $|g_0|^{1/2}=a^{2n/(n-2)}$. Differentiating
\[
\bigl(L_{A_{g_t}}+k^2|g_t|^{1/2}\bigr)(u\circ\varphi_t^{-1})=0\quad\text{near }\supp X
\]
at $t=0$ proves \r{74}. Testing against $v$ and integrating by parts proves \r{73}, since $(L_{A_0}+k^2|g_0|^{1/2})v=0$ near the support of $X\cdot\nabla u$. Both identities extend by $H^1$ density to the stated classes of vector fields, as in \cite[Lemma~3.13]{MS26}.
\end{proof}

We can call $\mathcal G_aX$ with $X\in H_0^1(M;\R^n)$ \emph{potential tensor fields}, by analogy with the X-ray geodesic transform of symmetric two-tensor fields. Then \r{73} is the analog of the fact that the geodesic X-ray transform of two-tensor fields vanishes on potential ones. 

The next two lemmas prepare the proof that the kernel of the linearization consists exactly of potential tensors $H=\mathcal G_aX$, with $X\in H_0^1(M;\R^n)$. Background complex geometric optics solutions yield a condition on $\widehat H$, which the next lemma converts into the decomposition $H=DX+DX^\top+\beta\e$. By \r{72}, this can be rewritten as $H=\mathcal G_aX+\sigma\e$. Lemma~\ref{lemma18} eliminates the scalar remainder $\sigma$, and we then show that $X$ vanishes outside $M$. In the completion of the proof, the limiting gauge condition forces $X=0$, contradicting the unit norm of the limit obtained by compactness.
\begin{lemma}[\!{\cite[Lemma~3.14]{MS26}}]\label{lemma17}
Let $H\in L^2(\R^n)$ be a real symmetric tensor with compact support, and assume that for almost every $\xi\ne0$ the restriction of $\widehat H(\xi)$ to $\xi^\perp$ is a scalar multiple of the Euclidean form. Then there are compactly supported $X\in H^1(\R^n;\R^n)$ and $\beta\in L^2(\R^n)$ such that
\begin{equation}\label{75}
H=DX+DX^\top+\beta\e.
\end{equation}
\end{lemma}

The next lemma is the Helmholtz version of \cite[Lemma~3.15]{MS26}. It proves injectivity of the linearization for conformal perturbations.
\begin{lemma}\label{lemma18}
Let $a>0$ be smooth on $\R^n$ and equal to one outside a ball, and set $g_0=a^{4/(n-2)}\e$. Suppose that $\sigma\in L^2(\R^n)$ has support in the interior of a closed ball $\B$ and
\begin{equation}\label{78}
\int_{\R^n}\sigma\left(a^2\nabla u\cdot\nabla v-\frac{nk^2}{n-2}|g_0|^{1/2}uv\right) \d x=0
\end{equation}
for every pair $u,v\in C^\infty(\R^n)$ satisfying $(\Delta_{g_0}+k^2)u=(\Delta_{g_0}+k^2)v=0$ on a neighborhood of $\B$. Then $\sigma=0$.
\end{lemma}

\begin{proof}
The product identity is now
\[
L_{A_0}(uv)=2a^2\nabla u\cdot\nabla v-2k^2|g_0|^{1/2}uv.
\]
Substituting this into \r{78} and integrating by parts yields
\begin{equation}\label{79}
\left\langle a^{-2}\left(L_{A_0}\sigma-\frac{4k^2}{n-2}|g_0|^{1/2}\sigma\right),(au)(av)\right\rangle=0.
\end{equation}
The distribution in this pairing belongs to $H^{-2}$ and has compact support. Using the background construction of Lemma~\ref{lemma10} with two derivatives on $\B$, the complex geometric optics limiting argument of \cite[Lemma~3.15]{MS26} shows that its Fourier transform vanishes. Hence
\[
L_{A_0}\sigma-\frac{4k^2}{n-2}|g_0|^{1/2}\sigma=0.
\]
Elliptic regularity makes $\sigma$ smooth. Since $\sigma$ has compact support, unique continuation (or the negativity of the associated quadratic form) implies $\sigma=0$.
\end{proof}

The next proposition identifies the kernel of the linearized Helmholtz DN map, expressed in conductivity variables, with the potential tensors. It proves injectivity of the full linearization modulo infinitesimal boundary-fixing diffeomorphisms. It is the Helmholtz version of \cite[Proposition~3.16]{MS26}; its $L^2$ formulation is needed for the limit in Lemma~\ref{lemma15}.

\begin{proposition}\label{pr2}
A real symmetric tensor $H\in L^2(M)$ satisfies
\begin{equation}\label{81}
\int_M\left(a^2H^{ij}(\partial_i u)\,(\partial_j v)-\frac{k^2}{n-2}|g_0|^{1/2}(\tr H)uv\right)\,\d x=0
\end{equation}
for all $u,v\in C^\infty(M)$ satisfying $(\Delta_{g_0}+k^2)u=(\Delta_{g_0}+k^2)v=0$ if and only if
\begin{equation}\label{82}
H=\mathcal G_aX\qquad\text{for some }X\in H_0^1(M;\R^n).
\end{equation}
\end{proposition}

\begin{proof}
The reverse implication is Lemma~\ref{lemma16}. For the forward implication, extend $H$ by zero to $\R^n$. Fix $\xi\ne0$ and orthonormal $p,q\in\xi^\perp$, and test \r{81} with the background complex geometric optics solutions with phases $\zeta_t,\eta_t$ used in \cite[Proposition~3.16]{MS26}. Since $\zeta_t+\eta_t=-\i\xi$, the additional zeroth-order term, after division by $t^2$, is $O(k^2t^{-2}\|H\|_{L^2})$ by \r{43}. Thus the same limit yields
\begin{equation}\label{83}
\widehat H(\xi)(p+\i q,p+\i q)=0.
\end{equation}
As in that proposition, the restriction of $\widehat H(\xi)$ to $\xi^\perp$ is scalar. Lemma~\ref{lemma17} supplies compactly supported $X,\beta$, chosen using only $H$, and
\begin{equation}\label{84}
H-\mathcal G_aX=\sigma\e,\qquad \sigma=\beta+\diver X+2X\cdot\nabla\log a.
\end{equation}
The scalar $\sigma$ belongs to $L^2$ and has compact support. For background Helmholtz solutions near a ball containing $M\cup\supp X\cup\supp\sigma$ in its interior, subtracting \r{73} from \r{81} gives \r{78}. Lemma~\ref{lemma18} therefore implies $\sigma=0$.

To prove that $X$ vanishes outside $M$, we vary the background there as in \cite[Proposition~3.16]{MS26}. For real $\varphi\in C_0^\infty(\R^n\setminus M)$, set $\tilde a=ae^\varphi$ and $\tilde g_0=\tilde a^{4/(n-2)}\e$. Both the conductivity and volume density agree with those of $g_0$ on $M$. Hence \r{81} holds for this background as well. Keeping the same $X,\beta$ and repeating the scalar argument with $\tilde a$ gives $H=\mathcal G_{\tilde a}X$, and therefore
\begin{equation}\label{85}
0=\mathcal G_{\tilde a}X-\mathcal G_aX=-2(X\cdot\nabla\varphi)\e.
\end{equation}
Since \r{85} holds for every such $\varphi$, the support argument of \cite[Proposition~3.16]{MS26} implies that $X=0$ outside $M$, including its bounded complementary components. This proves \r{82}.
\end{proof}

\subsection{Completion of the proof}\label{sec3j}
\begin{proof}[Proof of Theorem~\ref{thm2}]
Suppose the conclusion fails. As in the proof of \cite[Theorem~2]{MS26}, apply Lemmas~\ref{lemma6}--\ref{lemma7} to a sequence of metrics $g_j\to g_0$ in $H^s(M)$ with $\Lambda_{g_j}(k)=\Lambda_{g_0}(k)$ which are not related to $g_0$ by boundary-fixing diffeomorphisms. Write $\overline{g}_j=(\Phi_j)_*g_j$, with perturbations $B_j$ as in \r{20}. Then $b_{s,j}\to0$ by \r{29}, and $B_j\ne0$, since $B_j=0$ would imply $\overline{g}_j=g_0$. Lemma~\ref{lemma15} yields, after passing to a subsequence,
\[
H_j=B_j/b_{2,j}\longrightarrow H\quad\text{in }L^2,\qquad \|H\|_{L^2}=1.
\]
Set $A_j=A_{\overline{g}_j}=a^2(\Id+B_j)$. By \r{35}, $A_j\to A_0$ and $|\overline{g}_j|^{1/2}\to|g_0|^{1/2}$ in $L^\infty(M)$. Since
\[
|\overline{g}_j|^{1/2}=|g_0|^{1/2}\det(\Id+B_j)^{1/(n-2)},
\]
Taylor's formula implies
\begin{equation}\label{86a}
\left\|\frac{|\overline{g}_j|^{1/2}-|g_0|^{1/2}}{b_{2,j}}-\frac{|g_0|^{1/2}}{n-2}\tr H_j\right\|_{L^2(M)}\le Cb_{\infty,j}\longrightarrow0.
\end{equation}
Thus the normalized volume-density differences converge to $|g_0|^{1/2}\tr H/(n-2)$ in $L^2(M)$.

Fix $u,v\in C^\infty(M)$ satisfying $(\Delta_{g_0}+k^2)u=(\Delta_{g_0}+k^2)v=0$, and let $u_j$ solve $(\Delta_{\overline{g}_j}+k^2)u_j=0$ with boundary value $u$. Then
\[
\bigl(L_{A_j}+k^2|\overline{g}_j|^{1/2}\bigr)(u_j-u) =-\diver(a^2B_j\nabla u)-k^2\bigl(|\overline{g}_j|^{1/2}-|g_0|^{1/2}\bigr)u,
\qquad (u_j-u)|_{\partial M}=0.
\]
By the assumption on $k^2$, the Dirichlet inverses are uniformly bounded from $H^{-1}(M)$ to $H_0^1(M)$, as in Lemma~\ref{lemma6}. Therefore
\begin{equation}\label{86}
\|u_j-u\|_{H^1(M)}\le Cb_{2,j},
\end{equation}
where $C$ may depend on $u$. Dividing the Alessandrini identity \r{19} by $b_{2,j}$, replacing $u_j$ by $u$ using \r{86}, and applying \r{86a}, we obtain
\begin{equation}\label{87}
\left|\int_M\left(a^2H_j^{ij}(\partial_i u)\,(\partial_j v)-\frac{k^2}{n-2}|g_0|^{1/2}(\tr H_j)uv\right)\,\d x\right| \le C(b_{2,j}+b_{\infty,j})\longrightarrow0.
\end{equation}
Passing to the limit in \r{87} proves \r{81} for $H$. Proposition~\ref{pr2} therefore implies $H=\mathcal G_aX$ for some $X\in H_0^1(M;\R^n)$.

Let $Z_j$ be the normal fields from Lemma~\ref{lemma6} and set $V_j=Z_j/b_{2,j}$. By \r{22}, this sequence is bounded in $H_0^1(M;\R^N)$, and \r{21}, \r{27} imply
\begin{equation}\label{88}
\bigl(L_{A_j}+k^2|\overline{g}_j|^{1/2}\bigr)V_j =-\diver(a^2H_j\nabla F) -k^2\frac{|\overline{g}_j|^{1/2}-|g_0|^{1/2}}{b_{2,j}}F,
\qquad DF^\top V_j=0.
\end{equation}
After passing to a further subsequence, $V_j\rightharpoonup V$ in $H_0^1$. Using the $L^\infty$ convergence of the coefficients, $H_j\to H$ in $L^2$, and \r{86a}, we pass to the weak formulation of \r{88} and obtain
\begin{equation}\label{89}
\begin{split}
\bigl(L_{A_0}+k^2|g_0|^{1/2}\bigr)V
&=-\diver(a^2H\nabla F)-\frac{k^2}{n-2}|g_0|^{1/2}(\tr H)F,\\
DF^\top V&=0.
\end{split}
\end{equation}
Each component of $F$ solves the background Helmholtz equation. Lemma~\ref{lemma16}, applied to $H=\mathcal G_aX$, shows that $-DFX\in H_0^1(M;\R^N)$ solves the same Dirichlet problem as $V$ in \r{89}. Since $k^2$ is not a Dirichlet eigenvalue of $-\Delta_{g_0}$, we have $V=-DFX$. The final argument is the same as in the proof of \cite[Theorem~2]{MS26}: normality implies $DF^\top DF\,X=0$, and the full rank of $DF$ from Proposition~\ref{prop:embedding} forces $X=0$. Hence $H=0$, contradicting $\|H\|_{L^2}=1$.
\end{proof}

\section{Proof of Corollary~\ref{cor2}}
By the strict domain monotonicity of the Dirichlet eigenvalues, we can choose a slightly larger concentric ball $B'\Supset\overline B$, if needed, so that $k^2$ is not a Dirichlet eigenvalue of $-\Delta_{g_0}$ in $B'$. The same holds for $g$ when $\eps_0$ is sufficiently small.

The next argument is standard in inverse scattering. Equality of the scattering matrices implies equality of the DN maps on $\partial B'$; see \cite[Section~1, p.~237]{Uhlmann00}. Indeed, Rellich's lemma shows that the scattering solutions with the same incident plane wave coincide outside $B$. Their Cauchy data therefore agree on $\partial B'$. By duality and unique continuation, the linear span of their Dirichlet traces is dense in $H^{1/2}(\partial B')$, which proves the assertion.  
See also \cite[Chapter~IV, Theorem~5.1]{Stefanov23}.

Theorem~\ref{thm2}, applied in $B'$, yields $g=\Phi^*g_0$ with $\Phi|_{\partial B'}=\Id$. Since $\Phi$ is an isometry fixing the boundary pointwise and preserving the interior, $D\Phi$ fixes tangent vectors and the inward unit normal on $\partial B'$. The inward normal geodesics for both metrics are radial segments in $B'\setminus\overline B$, so uniqueness for the geodesic equation implies $\Phi(r,\omega)=(r,\omega)$ there, in polar coordinates about the common center of the balls. Extending $\Phi$ by the identity outside $B'$ proves the corollary.


\end{document}